\documentclass[10pt,oneside,a4paper]{amsart}
\usepackage{geometry}
\usepackage[english]{babel}
\usepackage{amsfonts}
\usepackage{amsmath}
\usepackage{amssymb}
\usepackage{graphicx}
\usepackage{caption}
\usepackage{subcaption}
\usepackage{amsthm}
\usepackage{eurosym}
\usepackage[dvipsnames,usenames]{color}
\usepackage{cite}
\usepackage{soul}
\usepackage{tikz}
\usepackage{pgf}
\usepackage{lmodern} %separación de palabras acentuadas
\usepackage[T1]{fontenc}
\title[Convexity properties and Hardy's uncertainty principle]{\textbf{Convexity properties of discrete Schr\"odinger evolutions and Hardy's Uncertainty principle}}
\author{Aingeru Fern\'andez-Bertolin}
\keywords{}
\address{A. Fern\'andez-Bertolin:Departamento de Matem\'aticas,  Universidad del Pa\'is Vasco UPV/EHU, apartado 644, 48080, Bilbao, Spain}
\email{aingeru.fernandez@ehu.es}
\usepackage{hyperref} % hipervinculos
\newtheorem{teo}{Theorem}[section]
\newtheorem{prop}{Proposition}[section]
\newtheorem{lem}{Lemma}[section]

\theoremstyle{remark}
\newtheorem{rem}{Remark}[section]
\DeclareMathOperator{\arcsinh}{arcsinh}
\newcommand{\omegat}{\tilde{\omega}}
\newcommand{\zt}{\tilde{z}}

\newcommand{\sumzd}{\sum_{j\in\mathbb{Z}^d}}

\begin{document}

\begin{abstract}
In this paper we give log-convexity properties for solutions to discrete Schr\"odinger equations with different discrete versions of Gaussian decay at two different times. For free evolutions, we use complex analysis arguments to derive these properties, while in a perturbative setting we use a preliminar log-convexity result in order to get these properties. Then, by proving a Carleman inequality we conclude, in one of the cases under study, a discrete version of Hardy's Uncertainty Principle.
\end{abstract}
\maketitle

\section{Introduction}
The aim of this paper is to show log-convexity properties for solutions to the free Schr\"odinger equation
\begin{equation}\label{eqsch}
\partial_tu_j=i\Delta_du_j=i\sum_{k=1}^d(u_{j+e_k}-2u_j+u_{j-e_k}),
\end{equation}
and also for solutions to a perturbed discrete Schr\"odinger equation
\begin{equation}\label{eqschpot}
\partial_tu_j=i(\Delta_du_j+V_ju_j),
\end{equation}
when $V$ is a time-dependent bounded potential.

Our motivation is the relation between log-convexity properties and Hardy's uncertainty principle (see \cite{dm,ss}),
\begin{equation}\label{hardy}
|f(x)|\le Ce^{-\alpha |x|^2},\ \ |\hat{f}(\xi)|\le Ce^{-\beta|\xi|^2}, \ \text{with }\alpha\beta>\frac{1}{4}\Rightarrow f\equiv0,
\end{equation}
and in the case $\alpha\beta=\frac14$ then $f(x)=Ce^{-\alpha |x|^2}$. This uncertainty principle can be understood as an amplification of Heisenberg's uncertainty principle
\[
\frac{2}{d}\left(\int_{\mathbb{R}^d}|xf(x)|^2\,dx\right)^{1/2}\left(\int_{\mathbb{R}^d}|\nabla f(x)|^2\,dx\right)^{1/2}\ge \int_{\mathbb{R}^d}|f(x)|^2\,dx,
\]
where the equality is attained if and only if $f(x)=C e^{-\alpha |x|^2}$ for $\alpha>0$. Moreover, writing a solution to the Schr\"odinger equation $\partial_tu=i\Delta u$ as
\[
u(x,t)=\frac{e^{i|x|^2/4t}}{(it)^{d/2}}\left(e^{i|\cdot|^2/4t}u_0\right)^\wedge\left(\frac{x}{2t}\right),
\] 
Hardy's uncertainty principle can be written in terms of solutions to the Schr\"odinger equation in a $L^2$ setting (see \cite{cp}) as follows: 
\[
\|e^{\alpha |x|^2}u(0)\|_{L^2(\mathbb{R}^d)}+\|e^{\beta |x|^2}u(1)\|_{L^2(\mathbb{R}^d)}<+\infty,\ \ \alpha\beta>\frac{1}{16} \Rightarrow u\equiv0,
\]
so this result states that a solution to the Schr\"odinger equation cannot decay too fast at two different times simultaneously. The classical proof of this uncertainty principle is based on complex analysis arguments (Phragm\'en-Lindel\"of's principle and Liouville's theorem), but there is a series of papers, \cite{ekpv1,ekpv2,ekpv3,ekpv4} and \cite{cekpv}, where the authors prove Hardy's uncertainty principle in this dynamical setting, considering solutions to perturbed Schr\"odinger equations and using real variable arguments. One of the main tools they use is precisely a log-convexity result that states that a solution to those equations with Gaussian decay at two different times preserves the Gaussian decay at any time in between. This process to prove Hardy's principle using real calculus starts in \cite{ekpv2} with a non-sharp result combining the log-convexity property with a Carleman inequality, and then in \cite{ekpv3} they use an iterative process to go from this preliminar result to the sharp result.

In the discrete setting, we studied in a previous paper, \cite{fb}, a version of Heisenberg's principle (see also \cite{an,cn,gg} for more references to this uncertainty principle) generated by the discretization of the position and momentum operators
\begin{equation}\label{ops}
\mathcal{S}^hu_j=jhu_j=(j_1h,\dots,j_dh)u_k,\ \ \ \ \ \mathcal{A}^hu_j=\left(\frac{u_{j+e_1}-u_{j-e_1}}{2h},\dots,\frac{u_{j+e_d}-u_{j-e_d}}{2h}\right),
\end{equation}
where $j\in\mathbb{Z}^d$, $e_k=(0,\dots,0,\overbrace{1}^k,0,\dots,0)$, for $k=1,\dots,d$, and related it to the discrete Schr\"odinger equation \eqref{eqsch} via a Virial identity. In this case, the minimizer $\omega$ (the analogous of the Gaussian function) is given in terms of modified Bessel functions of the first kind
\[
I_m(x)=\frac{1}{\pi}\int_0^\pi e^{z\cos\theta}\cos(m\theta)\,d\theta,\ \ m\in\mathbb{Z},\ \ \omega=(\omega_j)_{j\in\mathbb{Z}^d}=\left(C_{h\lambda}I_j\left(\frac{1}{2\lambda h^2}\right)\right)_{j\in\mathbb{Z}^d}.
\]

Moreover, in the paper we saw how we can recover the Gaussian $e^{\lambda |x|^2}$ from the minimizer $\omega$ as the mesh step tends to zero.

On the other hand, using complex analysis arguments, we gave in \cite{fb2} a discrete version of Hardy's principle in one dimension, similar to \eqref{hardy}, that can be written in terms of solutions to the discrete free Schr\"odinger equation as in the classical case,
\begin{equation}\label{hardyd}
|u_j(0)|\le I_j(\alpha),\ \ |u_j(1)|\le I_j(\beta),\ \ \alpha+\beta<2\Rightarrow u\equiv0.
\end{equation}

It is natural to think then that we should be able to prove this result following the approach in \cite{ekpv2}, so the first thing we need to understand is the analogous version of the log-convexity property stated in \cite{ekpv2}. Notice that in the discrete setting we can give many discrete versions of Gaussian decay. Looking at \eqref{hardyd}, the first interpretation to play the role of $e^{\lambda |x|^2}$ one can think of is the inverse of the minimizer $\omega$. However, we can understand the function $e^{\lambda |x|^2}$ as the solution to the adjoint equation of $\nabla f+2\lambda x f=0$, the equation satisfied by the Gaussian. If we do the same using the operators \eqref{ops}, it is easy to check that now the weight is given in terms of modified Bessel functions of the second kind 
\[
K_m(x)=\int_0^\infty e^{-x\cosh t}\cosh(m t)\,dt.
\]

On the other hand, we can simplify more the discrete interpretation of Gaussian decay, just using the weight function $e^{\lambda |j|^2}$.

Here we give two different methods in order to prove that solutions to the discrete Schr\"odinger equation with those discrete versions of Gaussian decay satisfy a log-convexity property. Formally we see that the log-convexity holds, but trying to justify these formal calculations is where we use different methods. First, by relating the discrete solution to a periodic function via Fourier series we can use complex analysis arguments in order to justify everything.  However, this method is only useful when considering solutions to the free case \eqref{eqsch}. On the other hand, for solutions to \eqref{eqschpot} we can give a preliminar log-convexity property, in the spirit of \cite{kpv}, using a linear exponential weight, which, by a simple fact, allows us to prove the log-convexity properties we want in the perturbative setting. The advantage of using the first method is that proving the log-convexity directly we can get some a priori estimates. These estimates were crucial in the continuous case, although in the discrete setting this does not seem to happen.

Once we have the log-convexity properties for the different weights, since we want to follow the approach in \cite{ekpv2}, we need a Carleman inequality in order to prove Hardy's uncertainty principle. Using modified Bessel functions, it is not clear how we should take the frequency function, but in this paper we prove that there is a Carleman inequality associated to $e^{\lambda |j|^2}$ for the discrete Schr\"odinger equation, analogous to the inequality in \cite{ekpv2}, that can be used to prove a non-sharp discrete version of Hardy's uncertainty principle. Hence, we give a discrete uncertainty principle whose proof looks like the proof in the continuous case.

Preparing this manuscript, we learned about a recent and independent result in this direction \cite{jlmp}. There, the authors also prove a sharp analog of Hardy's uncertainty principle in the discrete setting, in terms of solutions to the $1d$ discrete free Schr\"odinger equation by using complex analysis arguments. To avoid the use of complex analysis, and to add a potential to the equation, they adapt the log-convexity approach in \cite{ekpv2}, getting also a non-sharp result in this case. Their result improves our version of Hardy's principle (Theorem 4.1 below), when we focus on the one dimensional case, although the proofs of the results are rather different, since here we follow more closely the approach in \cite{ekpv2}.

The paper is organized as follows: In Section 2 we give log-convexity properties using the weights discussed above for discrete free Schr\"odinger evolutions, so, in order to justify the formal calculations, we use tools of complex analysis. Then, in Section 3, we add a potential and prove a result using a linear exponential weight, concluding from this result the log-convexity properties of Section 2, now in a perturbative setting. Finally, in Section 4 we prove the discrete version of Hardy's principle by means of a Carleman inequality.

In our previous papers, we study the discrete Schr\"odinger equation with the mesh step $h$, so that when $h$ tends to zero the solution to the discrete equation converges to the solution to the continuous equation. Here, since we are not going to study convergence of the results when $h$ tends to zero, we fixed, just for simplicity, $h=1$, although all the results can be written introducing this parameter.

\section{Log-convexity properties of discrete free Schr\"odinger evolutions}

To begin with, we consider that the solution to \eqref{eqsch} decays when we multiply it by the inverse of the discrete minimizer in \cite{fb}. Then we have the following result:

\begin{teo}
Assume $u=(u_j)_{j\in\mathbb{Z}^d}$ is a solution to the $d$-dimensional free Schr\"odinger equation \eqref{eqsch} which satisfies
\[
\sum_{j\in\mathbb{Z}^d}\frac{1}{\omega_j^2}|u_j(0)|^2+\sum_{j\in\mathbb{Z}^d}\frac{1}{\omega_j^2}|u_j(1)|^2<+\infty,
\]
where $\omega_j=C_{d,\lambda}\prod_{k=1}^d I_{j_k}(1/2\lambda)$ for some $\lambda>0$. Then
\[
F(t)=\sum_{j\in\mathbb{Z}^d}\frac{1}{\omega_j^2}|u_j(t)|^2 \text{ is logarithmically convex in } 0\le t\le 1.
\]
\end{teo}

In order to prove that $F(t)$ is log-convex in the interval $[0,1]$, we want to use the following lemma (see \cite{ekpv2}).
\begin{lem}
Assume that $f(t)$ satisfies $\partial_tf=\mathcal{S}f+\mathcal{A}f$, where $\mathcal{S}$ is a symmetric operator and $\mathcal{A}$ is a skew-symmetric operator (both independent of $t$). If $[\mathcal{S},\mathcal{A}]\ge0$, then $H(t)=\langle f,f\rangle$ is logaritmically convex. 
\end{lem}

\begin{proof}[Proof of Theorem 2.1]

Formally, we set $f_j(t)=\frac{u_j(t)}{\omega_j}$ and it is easy to check that $\partial_tf_j=\mathcal{S}f_j+\mathcal{A}f_j$ where
\[\begin{aligned}
&\mathcal{S}g_j=\frac{i}{2}\sum_{k=1}^d\left(\left(\frac{\omegat_{j_k+1}}{\omegat_{j_k}}-\frac{\omegat_{j_k}}{\omegat_{j_k+1}}\right)g_{j+e_k}+\left(\frac{\omegat_{j_k-1}}{\omegat_{j_k}}-\frac{\omegat_{j_k}}{\omegat_{j_k-1}}\right)g_{j-e_k}\right),\\
&\mathcal{A}g_j=\frac{i}{2}\sum_{k=1}^d\left(\left(\frac{\omegat_{j_k+1}}{\omegat_{j_k}}+\frac{\omegat_{j_k}}{\omegat_{j_k+1}}\right)g_{j+e_k}-4g_j+\left(\frac{\omegat_{j_k-1}}{\omegat_{j_k}}+\frac{\omegat_{j_k}}{\omegat_{j_k-1}}\right)g_{j-e_k}\right),
\end{aligned}\]
where we define, for $n\in\mathbb{Z}$, $\tilde{\omega}_n=I_n(1/2\lambda)$. Hence, since $F(t)=\langle f,f\rangle$, in order to use the lemma we need to show that $[\mathcal{S},\mathcal{A}]=\mathcal{S}\mathcal{A}-\mathcal{A}\mathcal{S}\ge 0.$

The commutator of these operators is given by
\[\begin{aligned}
(\mathcal{S}\mathcal{A}-\mathcal{A}\mathcal{S})g_j=&\frac{1}{2}\sum_{k=1}^d\left(\left(\frac{\omegat_{j_k}\omegat_{j_k+2}}{\omegat_{j_k+1}^2}-\frac{\omegat_{j_k+1}^2}{\omegat_{j_k}\omegat_{j_k+2}}\right)g_{j+2e_k}+\left(\frac{\omegat_{j_k}\omegat_{j_k-2}}{\omegat_{j_k-1}^2}-\frac{\omegat_{j_k-1}^2}{\omegat_{j_k}\omegat_{j_k-2}}\right)g_{j-2e_k}\right)\\
&+\sum_{k=1}^d\left(\frac{8j_k^2\lambda^2\omegat_{j_k}^4}{\omegat_{j_k+1}^2\omegat_{j_k-1}^2}-8j_k^2\lambda^2-\frac{\omegat_{j_k-1}\omegat_{j_k+1}}{\omegat_{j_k}^2}+\frac{\omegat_{j_k}^2}{\omegat_{j_k-1}\omegat_{j_k+1}}\right)g_j.
\end{aligned}\]

Thus, the expression we want to be positive is $\langle[\mathcal{S},\mathcal{A}]g,g\rangle$, which after some calculations is
\[\begin{aligned} 
&\sumzd\sum_{k=1}^d\left(\frac{\omegat_{j_k}^2}{2\omegat_{j_k+1}\omegat_{j_k-1}}-\frac{\omegat_{j_k-1}\omegat_{j_k+1}}{2\omegat_{j_k}^2}\right)\left|g_{j+e_k}-g_{j-e_k}\right|^2\\&+\sumzd\sum_{k=1}^d\left(\frac{8j_k^2\lambda^2\omegat_{j_k}^4}{\omegat_{j_k+1}^2\omegat_{j_k-1}^2}-8j_k^2\lambda^2+\frac{\omegat_{j_k}^2}{\omegat_{j_k-1}\omegat_{j_k+1}}-\frac{\omegat_{j_k-1}\omegat_{j_k+1}}{\omegat_{j_k}^2}-\frac{\omegat_{j_k-1}^2}{2\omegat_{j_k-2}\omegat_{j_k}}\right.\\&\left.+\frac{\omegat_{j_k-2}\omegat_{j_k}}{2\omegat_{j_k-1}^2}-\frac{\omegat_{j_k+1}^2}{2\omegat_{j_k}\omegat_{j_k+2}}+\frac{\omegat_{j_k+2}\omegat_{j_k}}{2\omegat_{j_k+1}^2}\right)|g_j|^2.
\end{aligned}\]

Notice that for each $k$, the expressions that appear multiplying $\left|g_{j+e_k}-g_{j-e_k}\right|^2$ and $|g_j|^2$ are exactly the same, so once we prove that this is positive in one dimension, it is straightforward to prove it in the general case, so we restrict ourselves to the one dimensional version of the commutator. The first sum is positive by Amos inequality, \cite[p.~269]{am} or \cite[(1.9)]{tn}: $I_j^2(x)-I_{j+1}(x)I_{j-1}(x)>0$ for $x>0$ and $j\ge -1$. Notice that since $j$ is an integer, $I_{-j}(x)=I_j(x)$. Hence, it remains to prove that the second sum is positive when $j\in\mathbb{N}\cup\{0\}$. To simplify, we will consider $x=\frac{1}{2\lambda}$ and divide the expression of the second sum by $4\lambda^2$. That implies that if we are able to prove the following property for modified Bessel functions:
\begin{equation}\begin{aligned}\label{eq1}
\frac{2j^2I_j^4(x)}{I_{j+1}^2(x)I_{j-1}^2(x)}&-2j^2+x^2\left(\frac{I_j^2(x)}{I_{j-1}(x)I_{j+1}(x)}-\frac{I_{j+1}(x)I_{j-1}(x)}{I_j^2(x)}+\frac{I_{j+2}(x)I_{j}(x)}{2I_{j+1}^2(x)}\right.\\&\left.-\frac{I_{j+1}^2(x)}{2I_{j}(x)I_{j+2}(x)}+\frac{I_{j}(x)I_{j-2}(x)}{2I_{j-1}^2(x)}-\frac{I_{j-1}^2(x)}{2I_{j-2}(x)I_{j}(x)}\right)>0,
\end{aligned}\end{equation}
for $j\in\mathbb{N}\cup 0$ and $x>0$, then the convexity will hold.

Since $f(x)=x^{1/2}I_j(x)$ satisfies the equation
\[
f''(x)-\frac{j^2-1/4}{x^2}f(x)-f(x)=0,
\]
we see two behaviors, when $x$ is large enough and $x$ is small enough (with respect to $j$). Note that in the case of the Bessel function of the first kind $J_j(x)$ there is a cancellation term in the equation that gives another behavior.

These behaviors are given by the asymptotics (first for $x$ small enough, then for $x$ large enough)
\[
I_j(x)\sim \frac{(x/2)^j}{j!},
\]
\[
I_j(x)\sim \frac{e^x}{\sqrt{2\pi x}}\left(1-\frac{4j^2-1}{8x}+\frac{(4j^2-1)(4j^2-9)}{2!(8x)^2}-\frac{(4j^2-1)(4j^2-9)(4j^2-25)}{3!(8x)^3}\right).
\]

Therefore, when $x$ is small enough (with respect to $j$)
\[
(\ref{eq1})\sim 2(2j+1)-x^2\frac{2j+1}{(j-1)j(j+1)(j+2)}\ge \frac{(2j+1)(2j^3+4j^2-3j-4)}{(j-1)(j+1)(j+2)}\ge 0,
\]
if $j\ge 2$. The cases $j=0,1$ follow a similar argument but the calculations are slightly different.

On the other hand, if $x$ is large enough the best way to treat (\ref{eq1}) is writing it in the form of a quotient of two expressions involving modified Bessel functions. In the denominator we will have a product of modified Bessel functions (which are positive functions). Moreover, since the degree in both numerator and denominator is the same, we can avoid the term $\frac{e^x}{\sqrt{2\pi x}}$ in the asymptotic expansion.

If we make all the calculations, and only consider the leading term in the expression, we see that it behaves like
\[
\frac{2+8j^2}{x^3}\ge0.
\]

Thus, heuristically we have seen that it makes sense to think that (\ref{eq1}) is positive.

To give a rigorous proof of this, we will use rational bounds for modified Bessel functions in order to reduce the positivity of our expression to the positivity of a quotient of two polynomials. We will need to treat separately three different cases. For the sake of readability, we avoid the calculations here because of the large numbers involved, since the degree of the polynomials involved is quite high. In order to manage all these computations, we use \textit{Mathematica} as a useful tool to substitute coefficients of the polynomials according to the regions we are studying in each case.

\textbf{First case:} $j\ge 17$, $0< x\le j^{3/2}$. In this region we use the following Tur\'anian estimate, which is an immediate consequence of \cite[Theorem 1]{ba}:
\[
\frac{j+1/2}{(j+1)\sqrt{x^2+(j+1/2)^2}}I_j^2(x)<I_j^2(x)-I_{j-1}(x)I_{j+1}(x)<\frac{1}{x+2}I_j^2(x).
\]
 
 Using these bounds, after some computations we see that the positivity of \eqref{eq1} depends on the positivity of an expression of the following kind,
\[
(1+2j)p(j,x)-(1+j)\sqrt{1+4j+4j^2+4x^2}q(j,x),
\] 
where $p$ and $q$ are polynomials of positive coefficients whose degrees, as polynomials in $x$, are 6 and 5 respectively. Writing the difference $(1+2j)^2p^2(j,x)-(1+j)^2(1+4j+4j^2+4x^2)q^2(j,x)$ as a polynomial in $x$ (of degree 12) whose coefficients are polynomials in $j$, we see that in this region \eqref{eq1} is positive. Indeed, we look at the sign of the coefficient of highest degree in $x$, and when it is negative we use $x^2\le j^3$ to reduce the degree of the polynomial. Iterating this argument when necessary we see that the polynomial is positive.

\textbf{Second case:} $0\le j\le 17, x>0$. Now we use the rational bounds in \cite[Theorem 2 and Theorem 3]{na} to treat this case. There, the author gives a method to generate upper and lower bounds for the ratio $\frac{I_{j+1}(x)}{I_j(x)}$ based on the completely monotonicity of the function $x^{-j}e^{-j}I_j(x)$. More precisely, he proves that there are rational polynomials such that
\begin{equation}\label{eqna}
L_{j,k,m}(x)<\frac{I_{j+1}(x)}{I_j(x)}<U_{j,k,m}(x),
\end{equation}
giving also a method to compute these polynomials $L$ and $U$. Now,  setting $k=5$ and $m=0$, we get a rational bound for \eqref{eq1}. We can easily check that the denominator of the quotient is positive, and, for $j\le 17$, all the coefficients of the polynomial that appears in the numerator (of degree 39 in $x$) of the quotient are positive, so \eqref{eq1} is positive in this case.

\textbf{Third case:} $j\ge 17, x\ge j^{3/2}$. Here we use again (\ref{eqna}), with $k=5$, $m=0$. As we have pointed out in the second case, the denominator of the rational bound is positive, but now when $j$ is large some negative coefficients appear in the numerator. Again, the way to prove the positivity is to collect the coefficients in $x$ powers and look at the sign of the coefficient of highest degree in $x$. Now the coefficient is positive if $j\ge 17$, and using $x^2\ge j^3$ we reduce the degree of the polynomial and start again this process.

Hence, the separate study of \eqref{eq1} in these three cases gives us the positivity of the commutator and by Lemma 2.1 we get the log-convexity of $F(t)$, formally. If we can justify this formal argument, then the theorem holds.

In order to justify all the calculations, what we want to check is that
\[
F_\alpha(t)=\sum_{j=-\infty}^\infty\frac{|u_j(t)|^2}{I_j^2(\alpha)} 
\]
is well defined for $t\in(0,1)$, provided that $F_\alpha(0)+F_\alpha(1)<+\infty$ and $\alpha>0$. Once we prove this, we conclude that the formal calculations are valid and then we have that the log-convexity holds.

As a first step, we prove a weaker result that says that a solution bounded at two times by the modified Bessel function is still bounded at any time between them. Then, from this $\ell^\infty$ result we get the $\ell^2$ result we want.
\begin{prop}
Assume that $u=(u_j)_j$ is a solution to \eqref{eqsch} when $d=1$ such that
\[
|u_j(0)|+|u_j(1)|\le CI_j(\alpha),\ \forall j\in\mathbb{Z},
\]
for some $\alpha>0$. Then there is $C>0$ such that, for $t\in(0,1)$, $|u_j(t)|\le C\left(\frac{1}{t}+\frac{1}{1-t}\right)\frac{I_j(\alpha)}{\sqrt{|j|}}$.
\end{prop}

\begin{proof}[Proof of Proposition 2.1]

We consider that $u_j(t)=\hat{f}(j,t)$ for a $2\pi$-periodic function $f$. The evolution of $f(x,t)$ is given by
\[
f(x,t)=e^{2it(\cos x-1)}f(x,0)=e^{2i(t-1)(\cos x-1)}f(x,1).
\]

Moreover, since $|\hat{f}(j,0)|\le CI_j(\alpha)$, we have that $f(x,0)$ is extended to an entire and $2\pi$-periodic function
\[
f(z,0)=\sum_{j=-\infty}^\infty \hat{f}(j,0)e^{ijz},\ \ z=x+iy,
\]
so $f(x,t)$ inherits these properties for all time. Furthermore, using that
\[
\sum_{j=-\infty}^\infty I_j(\alpha)e^{-jy}=e^{\alpha \cosh y},
\]
we conclude that $|f(z,0)|+|f(z,1)|\le Ce^{\alpha \cosh y}$ for all $z=x+iy$. Hence
\[\begin{aligned}
|f(x-iy,t)|\le\left\{\begin{array}{l} Ce^{-2t\sin x\sinh y+\alpha\cosh y},\\ Ce^{-2(t-1)\sin x\sinh y+\alpha\cosh y}.\end{array}\right.
\end{aligned}\]

Since we want something that behaves better than $e^{\alpha\cosh y}$, we are going to use, for $y\ge0$, the first line when $\sin x$ is positive, that is, when $x\in[0,\pi].$ On the other hand the second line will be useful when $\sin x$ is negative, that is, when $x\in[-\pi,0]$.

We have to distinguish between $j$ positive and $j$ negative, although the procedure we follow is the same. The quantity we have to look at is
\[
\hat{f}(j,t)=\int_{-\pi}^\pi f(x,t)e^{-ijx}\,dx,
\]
and we have to see that this quantity is controlled by $I_j(\alpha)$. For $j$ positive, we integrate the function over the contour described in Figure \ref{fig1}, observing that the integral over the vertical lines vanishes due to the periodicity. Thus we see that, thanks to Cauchy's theorem

\begin{figure}
\centering
\begin{tikzpicture}[domain=0:4]
\draw[-, line width=2pt] (8,0) node[left] {$-\frac{\pi}{2}-\theta-iy$} -- (10,0);
\draw[->, line width=2pt] (8,3) node[left] {$-\frac{\pi}{2}-\theta$} -- (10,3) ;
\draw[<-, line width=2pt] (10,0)  -- (12,0) node[right] {$\frac{3\pi}{2}-\theta-iy$};
\draw[-,line width=2pt] (10,3)  -- (12,3) node[right] {$\frac{3\pi}{2}-\theta$};
\draw[->,line width=2pt] (8,0) --(8,1.5);
\draw[-,line width=2pt] (8,1.5) --(8,3);
\draw[->,line width=2pt] (12,3) --(12,1.5);
\draw[-,line width=2pt] (12,1.5) --(12,0);
\end{tikzpicture}
\caption{Contour integration}\label{fig1}
\end{figure}
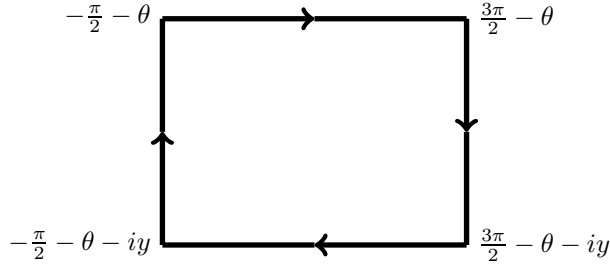
\[
\hat{f}(j,t)=\int_{-\pi}^{\pi}f(x-iy,t)e^{-ij(x-iy)}\,dx,\ \ \forall y>0.
\]

Now we split up the integral in order to use the bounds for $|f(x-iy,t)|$. Thus,
\[
|\hat{f}(j,t)|\le Ce^{\alpha\cosh y-jy}\left(\int_{-\pi}^{0}e^{-2(t-1)\sinh y\sin x}\,dx+\int_{0}^{\pi}e^{-2t\sinh y\sin x}\,dx\right).
\]

We can write each integral as a difference between a modified Bessel function of the first kind and a modified Struve function, both of order zero, having that
\[
|\hat{f}(j,t)|\le C e^{\alpha\cosh y-jy}\pi\big(I_0(2(1-t)\sinh y)-{\bf L}_0(2(1-t)\sinh y)+I_0(2t\sinh y)-{\bf L}_0(2t\sinh y)\big).
\]

Following the theory in \cite[\textsection 10.42]{wa}, we see that
\[
\pi(I_0(s)-{\bf L}_0(s))=\frac{2}{s}+R, \text{ where } |R|\le \frac{16}{s^3}.
\]

Hence, using this expression we obtain that for $y$ large enough,
\[
|\hat{f}(j,t)|\le C e^{\alpha\cosh y-jy}\left(\frac{3}{2(1-t)\sinh y}+\frac{3}{2t\sinh y}\right)\le C_{t}\frac{e^{\alpha\cosh y-jy}}{2\sinh y}.
\]

On the other hand, from \cite[Ch. 10, \textsection 7]{ol}, we have
\begin{equation}\label{asymb}
\left|\frac{\sqrt{2\pi j}(1+\alpha^2/j^2)^{1/4}I_j(\alpha)}{e^{j\sqrt{1+\alpha^2/j^2}-j \arcsinh(j/\alpha)}}-1\right|\le\frac{3}{5j},
\end{equation}
so, for $j$ large enough we have
\[
I_j(\alpha)\ge \frac{e^{j\sqrt{1+\alpha^2/j^2}-j \arcsinh(j/\alpha)}}{2\sqrt{2\pi j}(1+\alpha^2/j^2)^{1/4}}.
\]

Thus, if we set $y=\arcsinh(j/\alpha)$ (that tends to infinity when $j$ tends to infinity), we have that for $j$ large enough,
\[
|\hat{f}(j,t)|\le C_{t}\frac{e^{\alpha\sqrt{1+j^2/\alpha^2}-j\arcsinh(j/\alpha)}}{2j/\alpha}=C_{t}\frac{e^{j\sqrt{1+\alpha^2/j^2}-j\arcsinh(j/\alpha)}}{2j/\alpha}\le C_{t,\alpha}\frac{I_j(\alpha)}{\sqrt{j}},
\]
since $\sqrt{j}$ and $\sqrt{j}(1+1/j^2)^{1/4}$ behave in the same way as $j$ grows. If $j$ is negative, we use the same argument but instead of integrating the function $f(x-iy,t)$ we integrate $f(x+iy,t)$ over a similar contour and then we take $y=\arcsinh(|j|/\alpha)$.
\end{proof}
Since $\frac{1}{j}$ is not a summable function, we cannot use this proposition directly to justify these calculations. Nevertheless, we have that this implies that
\[
F_{\alpha+\epsilon}(t)=\sum_{j=-\infty}^\infty\frac{|u_j(t)|^2}{I_j^2(\alpha+\epsilon)}\le C\sum_{j=-\infty}^\infty\frac{I_j^2(\alpha)}{I_j^2(\alpha+\epsilon)}<+\infty.
\]

That the last sum is finite can be seen using the bounds in \cite[(4)]{ro}, and the same can be done for the time derivatives of $F_{\alpha+\epsilon}$. Hence we have that for $F_{\alpha+\epsilon}$ the formal calculations are correct, so it is a log-convex function for all $\epsilon>0$. Notice that since the constant in Proposition 2.1 blows up at $t=0$ and $t=1$, first we prove the log-convexity of $F_{\alpha+\epsilon}$ in an interval of the form $[t_0,t_1]\subset (0,1)$ and then, by using the convolution expression for the solution,
\[
u_j(t)=e^{-2it}\sum_{m=-\infty}^\infty u_m(0) I_{k-m}(2it)=e^{-2i(t-1)}\sum_{m=-\infty}^\infty u_m(1) I_{k-m}\big(2i(t-1)\big),
\]
we can let $t_0$ tend to 0 and $t_1$ tend to 1 to conclude the log-convexity in $[0,1]$ In other words, we have that
\[\begin{aligned}
\sum_{j=-\infty}^\infty\frac{|u_j(t)|^2}{I_j^2(\alpha+\epsilon)}&\le \left(\sum_{j=-\infty}^\infty \frac{|u_j(0)|^2}{I_j^2(\alpha+\epsilon)}\right)^{1-t}\left(\sum_{j=-\infty}^\infty \frac{|u_j(1)|^2}{I_j^2(\alpha+\epsilon)}\right)^{t}\\&\le  \left(\sum_{j=-\infty}^\infty \frac{|u_j(0)|^2}{I_j^2(\alpha)}\right)^{1-t}\left(\sum_{j=-\infty}^\infty \frac{|u_j(1)|^2}{I_j^2(\alpha)}\right)^{t}.
\end{aligned}
\]

Finally, by Fatou's lemma,
\[
\sum_{j=-\infty}^\infty\frac{|u_j(t)|}{I_j^2(\alpha)}\le \lim_{\epsilon\rightarrow0}\sum_{j=-\infty}^\infty\frac{|u_j(t)|^2}{I_j^2(\alpha+\epsilon)}\le  \left(\sum_{j=-\infty}^\infty \frac{|u_j(0)|^2}{I_j^2(\alpha)}\right)^{1-t}\left(\sum_{j=-\infty}^\infty \frac{|u_j(1)|^2}{I_j^2(\alpha)}\right)^{t},
\]
so the theorem holds. The same method can be used to justify the formal calculations in the general $d$-dimensional case.
\end{proof}
Now, as we have pointed out in the introduction, we have  other interpretations of Gaussian decay, so let us consider the solution to the adjoint equation that solves the modified Bessel function $I_j(x)$,
\[
\lambda j_k z_j-(z_{j+e_k}-z_{j-e_k})=0,\ \ j\in\mathbb{Z}^d.
\]

It is a simple computation to check that now the weight is given in terms of modified Bessel functions of the second kind $K_j(x)$. Using this weight, we have the following result:

\begin{teo}
Assume $u=(u_j)_{j\in\mathbb{Z}^d}$ is a solution to \eqref{eqsch} which satisfies
\begin{equation}\label{decaybessk}
\sumzd z_j^2|u_j(0)|^2+\sumzd z_j^2|u_j(1)|^2<+\infty,
\end{equation}
where $z_j=C_{d,\lambda}\prod_{k=1}^d K_{j_k}(1/2\lambda)$ for some $\lambda>0$. Then
\[
H(t)=\sumzd z_j^2|u_j(t)|^2 \text{ is logarithmically convex}.
\]
\end{teo}

As before, we are going to prove the log-convexity formally. In order to justify the calculations, we can argue in the same fashion as in Theorem 2.1, now proving the following one dimensional result (whose proof is based on the same arguments that we have used to prove Proposition 2.1):

\begin{prop}
Assume that a solution to the 1d discrete Schr\"odinger equation \eqref{eqsch} satisfies, $\forall j\in\mathbb{Z}$, $K_j(\alpha)(|u_j(0)|+|u_j(1)|)<C$, for some $C>0$ and $\alpha>0$. Then, for $t\in(0,1)$ we have
\[
K_j(\alpha)|u_j(t)|\le C_{\alpha}\left(\frac1t+\frac{1}{1-t}\right)\frac{1}{\sqrt{|j|}}, \ \text{if}\ j\ \text{is large enough}.
\]
\end{prop}

\begin{proof}[Proof of Theorem 2.2]

Again, we set $f_j=z_ju_j$ and carry out all the process to compute $[\mathcal{S},\mathcal{A}]$ in this case. We define $\zt_n=K_n(1/2\lambda)$ for $n\in\mathbb{Z}$, noticing that in the previous theorem we have done this for the inverse of $\zt_j$. Then in this case we have
\[\begin{aligned}
&\langle\left[\mathcal{S},\mathcal{A}\right]g,g\rangle=\sumzd\sum_{k=1}^d\left(\frac{\zt_{j_k+1}\zt_{j_k-1}}{2\zt_{j_k}^2}-\frac{\zt_{j_k}^2}{2\zt_{j_k+1}\zt_{j_k-1}}\right)|g_{j+e_k}-g_{j-e_k}|^2\\&+\sumzd\sum_{k=1}^d\left(\frac{\zt_{j_k+1}^2+\zt_{j_k-1}^2}{2\zt_{j_k}^2}-\frac{\zt_{j_k}^2}{2\zt_{j_k-1}^2}-\frac{\zt_{j_k}^2}{2\zt_{j_k+1}^2}+\frac{\zt_{j_k+1}^2}{2\zt_{j_k}\zt_{j_k+2}}-\frac{\zt_{j_k+2}\zt_{j_k}}{2\zt_{j_k+1}^2}+\frac{\zt_{j_k-1}^2}{2\zt_{j_k}\zt_{j_k-2}}-\frac{\zt_{j_k-2}\zt_{j_k}}{2\zt_{j_k-1}^2}\right)|g_j|^2.
\end{aligned}\]

As before, we only need to prove that the commutator is positive in one dimension. The first sum is positive due to the symmetry $K_{-j}(x)=K_j(x)$ for $j\in\mathbb{N}$ and the inequality $K_j^2(x)<K_{j-1}(x)K_{j+1}(x)$, valid for $j\ge0$ and $x>0$.

On the other hand, the positivity of the coefficients in the second sum is not straightforward. For simplicity, we define $\Lambda_j(x)$ as the $j$-th coefficient in the second sum,
\[\begin{aligned}
\Lambda_j(x)&=\frac{K_{j+1}^2(x)+K_{j-1}^2(x)}{2K_j^2(x)}-\frac{K_j^2(x)}{2K_{j-1}^2(x)}-\frac{K_j^2(x)}{2K_{j+1}^2(x)}+\frac{K_{j+1}^2(x)}{2K_j(x)K_{j+2}(x)}\\&-\frac{K_{j+2}(x)K_j(x)}{2K_{j+1}^2(x)}+\frac{K_{j-1}^2(x)}{2K_j(x)K_{j-2}(x)}-\frac{K_{j-2}(x)K_j(x)}{2K_{j-1}^2(x)},
\end{aligned}\]
and we need to prove that $\Lambda_j(x)>0$ for $j\in\mathbb{N}\cup0$ and $x>0$ . We start proving separately the cases $j=0,1$ and then we study the case $j\ge2.$

To prove that $\Lambda_0(x)\ge0$ we see that this is equivalent to prove  that $\frac{K_1^4(x)}{K_0^3(x)K_2(x)}>1$, which is a consequence of the estimates given in \cite[Theorem 2]{ba}.

If $j=1$, we consider two cases $0<x\le1.1$ and $x\ge1.1$. In the first case we use different estimates for each term in $\Lambda_1(x)$
\[
\frac{K_2(x)}{K_1(x)}>1+\frac{3}{2x},\ \ \frac{K_2^2(x)}{K_1(x)K_3(x)}>\frac{1}{1+\frac{1}{x}},\ \ \frac{K_0(x)}{K_1(x)}>\frac{x}{\frac{1}{2}+\sqrt{x^2+\frac{1}{4}}}.
\]

 The first estimate comes from the differentiation of the function $e^xK_0(x)$, which is a completely monotonic function according to \cite[Theorem 5]{ms} and the third estimate was proved in \cite[Theorem 1]{segu}. As a result, we obtain that $\Lambda_1(x)$ is positive if  $p(x)-\sqrt{1+4x^2}q(x)$ is positive, where $p$ and $q$ are polynomials of degree 6 and 5 in $x$ respectively. As we have done in the previous result for the modified Bessel functions $I_j(x)$, studying the sign of $p^2(x)-(1+4x^2)q(x)^2$, which is a polynomial of degree 11, we obtain the positivity in the first region.
 
In the second region, we consider the following estimates,
\[
\frac{K_2(x)}{K_1(x)}>\frac{8+\frac{8}{x}}{8-\frac{4}{x}+\frac{3}{x^2}},\ \ \frac{K_2(x)}{K_3(x)}>\frac{8+\frac{8}{x}}{8+\frac{28}{x}+\frac{35}{x^2}},\ \ \frac{K_0^2(x)}{K_1^2(x)}>\frac{1}{1+\frac{1}{x}+\frac{1}{4x^3}}.
\]

In this case we use the completely monotonicity of $e^xK_1(x)$ to prove the first and second estimate. The third is given in \cite[Theorem 2]{ba}. Again, this gives a rational bound for $\Lambda_1(x)$ and it is easy to see that both numerator and denominator are positive if $x\ge1.1$. In the case of the denominator, it can be written as a product of positive polynomials. For the numerator, we obtain a polynomial of degree 11 with some negative coefficients, but using $x\ge 1.1$ they can be hidden in the positive coefficients.

If $j\ge2$ we use the recurrence of modified Bessel functions of the second kind
\[
K_{j+1}(x)-K_{j-1}(x)=2\frac{j}{x}K_j(x)
\]
to rewrite $\Lambda_j(x)$ as
\[\begin{aligned}
x^2\Lambda_j(x)=&2j^2-\frac{2j^2K_j^4(x)}{K_{j-1}^2(x)K_{j+1}^2(x)}+x^2\left(\frac{K_{j+1}(x)K_{j-1}(x)}{K_j^2(x)}-\frac{K_j^2(x)}{K_{j+1}(x)K_{j-1}(x)}\right.\\&\left.+\frac{K_{j-1}^2(x)}{2K_{j}(x)K_{j-2}(x)}-\frac{K_{j}(x)K_{j-2}(x)}{2K_{j-1}^2(x)}+\frac{K_{j+1}^2(x)}{2K_{j}(x)K_{j+2}(x)}-\frac{K_{j}(x)K_{j+2}(x)}{2K_{j+1}^2(x)}\right).
\end{aligned}\]

As we have done in the case $j=1$, we split up $x>0$ in two regions.

First, if $x\ge \frac{3j}{2}$, we use the following estimate, given in \cite[Theorem 2]{ba}:
\[
\frac{1}{1+\frac{1}{x}}\le \frac{K_v^2(x)}{K_{v-1}K_{v+1}}\le\frac{1}{1+\frac{1}{x}-\frac{v^2-\frac{1}{4}}{x^3}},\ \ \text{for }v>1/2.
\]

After using these estimates, we obtain that the positivity of $\Lambda_j$ depends on the positivity of a polynomial of degree 7 in the variable $x$. Studying the sign of the coefficient of highest degree (which is going to be positive) and using $x\ge 3j/2$ to reduce the degree of the polynomial, we obtain that in this region $\Lambda_j(x)\ge0$ after some iterations of this argument. Finally, if $0<x\le\frac{3j}{2}$, we change the upper bound, using \cite[Corollary 1]{segu},
\[
\frac{1}{1+\frac{1}{x}}\le \frac{K_v^2(x)}{K_{v-1}K_{v+1}}\le\frac{1}{1+\frac{1}{v-\frac{1}{2}+\sqrt{x^2+(v-\frac{1}{2})^2}}},\ \ \text{for }v>1/2.
\]

Now the positivity of $\Lambda_j$ depends on the positivity of an expression of the kind $p_j(x)+\sqrt{1-4j+4j^2+4x^2}q_j(x)$, where $p_j$ and $q_j$ are again polynomials of degree 5 and 4 in $x$. On the one hand, using the same strategy as before, we can see that $p_j$ and $q_j$ are positive polynomials if $j\ge 2,\ 0<x\le j$. When $x\ge j$ we compute the difference $p_j^2-(1-4j+4j^2+4x^2)q_j^2$, which is a polynomial of degree 9 in $x$, and we see that this difference is positive (in order to use the same argument as before, we have to distinguish $j\le x\le 4j/3$ and $4j/3\le x\le 3j/2$). Thus, there is no change of sign in  $p_j(x)+\sqrt{1-4j+4j^2+4x^2}q_j(x)$ and this implies that $\Lambda_j(x)\ge0$ as we want.

This completes the proof of $\Lambda_j(x)\ge0$, $\forall j\in\mathbb{N}\cup0,\ x>0$, and therefore we have that the commutator of the operators $\mathcal{S}$ and $\mathcal{A}$ is positive, giving as a result the log-convexity of the desired quantity provided that all the quantities involved are finite, by using Proposition 2.2.
\end{proof}
We can simplify more the discrete interpretation of the Gaussian decay, and use the weight function $e^{\lambda |j|^2}$, having the following result:

\begin{teo}
Assume $u=(u_j)_{j\in\mathbb{Z}^d}$ is a solution to the equation \eqref{eqsch} which satisfies
\begin{equation}\label{decaygauss}
\sumzd e^{2\lambda |j|^2}|u_j(0)|^2+\sumzd e^{2\lambda |j|^2}|u_j(1)|^2<+\infty,
\end{equation}
for some $\lambda>0$. Then
\[
G(t)=\sumzd e^{2\lambda |j|^2}|u_j(t)|^2 \text{ is logarithmically convex}.
\]
\end{teo}

\begin{proof}
Formally, we consider $f_j=e^{\lambda |j|^2}u_j$ and compute the operators $\mathcal{S}$ and $\mathcal{A}$ so that we can apply Lemma 2.1. In this case
\begin{equation}\label{gausw}
\langle [\mathcal{S},\mathcal{A}]f,f\rangle=\sinh(2\lambda)\sumzd\sum_{k=1}^d|f_{j+e_k}-f_{j-e_k}|^2+2\sinh(2\lambda)\sumzd\sum_{k=1}^d \big(\cosh(4\lambda j_k)-1\big)|f_j|^2\ge0.
\end{equation}

In order to justify the formal calculations we need again an $\ell^\infty$ result in one dimension analogous to Proposition 2.1 and Proposition 2.2:
\begin{prop}
Assume that a solution to the discrete Schr\"odinger equation satisfies, $\forall j\in\mathbb{Z},\ |u_j(0)|+|u_j(1)|<Ce^{-\alpha j^2}$, for some $C>0$ and $\alpha>0$. Then, for $t\in(0,1)$ we have $|u_j(t)|\le C_1 e^{-\alpha j^2}$ for all $j\in\mathbb{Z}$, with $C_1$ not depending on $t$.
\end{prop}

\begin{rem}
When considering modified Bessel functions, the constant blows up at time $t=0$ and $t=1$, and this is why first we have to prove the log-convexity in $[t_0,t_1]\subset (0,1)$ and then study what happens if $t_0$ tends to 0 and $t_1$ tends to 1. In this case, the constant is independent of $t$ and we can avoid this step in the justification.
\end{rem}

\begin{proof}[Proof of Proposition 2.3]
Consider that $u_j(t)=\hat{f}(j,t)$. As in Proposition 2.1, thanks to the decay conditions we can extend $f(x,0)$ and $f(x,1)$ as entire functions and they are $2\pi-$periodic. This means that at time $t$, $f(z,t)$ is $2\pi-$periodic and entire, and we recall that it is given by
\begin{equation}\label{eqfga}
f(z.t)=f(x+iy,t)=\left\{\begin{array}{l}e^{2it(\cos x\cosh y-i\sin x\sinh y-1)}f(z,0),\\ e^{2i(t-1)(\cos x\cosh y-i\sin x\sinh y-1)}f(z,1).\end{array}\right.
\end{equation}

Moreover, using Poisson's summation formula
\[
|f(z,0)|\le \sum_{j=-\infty}^\infty |u_j(0)|e^{-j y}\le e^{y^2/4\alpha}\sum_{j=-\infty}^\infty e^{-\alpha(j+y/2\alpha)^2}\le C_\alpha e^{y^2/4\alpha},
\]
having the same estimate for $|f(z,1)|$. Now, if $y\ge 0$ we can write $f(z,t)$ using the first and the second line in \eqref{eqfga} in order to have that $|f(z,t)|\le C_\alpha e^{y^2/(4a)}.$ Then, by Cauchy's theorem, if $j\le0$, we have for all $y>0$ that,
\[
u_j(t)=\int_{-\pi}^\pi f(x+iy,t)e^{-ij(x+iy)}\,dx\Rightarrow |u_j(t)|\le C_\alpha e^{j y+y^2/4\alpha}.
\]

Finally, we set $y=-2\alpha j$, so $|u_j(t)|\le C_\alpha e^{-2\alpha j^2+4\alpha^2 j^2/4\alpha}=C_\alpha e^{-\alpha j^2}.$ If $j\ge 0$ we can argue in the same fashion, changing the contour of integration.
\end{proof}

This proposition proves that under the hypotheses, the formal calculations are valid for $G_{\lambda-\epsilon}=\sum_j e^{2(\lambda-\epsilon)j^2}|u_j(t)|^2$, and by Fatou's lemma we conclude the result.
\end{proof}

In the three cases $\langle [\mathcal{S},\mathcal{A}]f,f\rangle$ is written as the sum of two positive terms. We can use this fact and \cite[(2.22)]{ekpv2} in order to give a-priori estimates for solutions to \eqref{eqsch}. For example, when considering the weight $e^{\lambda|j|^2}$, just by getting rid of the first sum in \eqref{gausw} we have that
\begin{equation}\label{ap1}
\int_0^1t(1-t)\sumzd\sum_{k=1}^d \big(\cosh(4\lambda j_k)-1\big)e^{2\lambda|j|^2}|u_j(t)|^2\,dt\le c(G(1)+G(0)),
\end{equation}
and in particular this implies that  in the interior we have more decay for the solution. On the other hand, using the formula
\[\begin{aligned}
\sumzd\sum_{k=1}^d&|f_{j+e_k}-f_{j-e_k}|^2+\sumzd\sum_{k=1}^d \big(\cosh(4\lambda j_k)-1\big)|f_j|^2\\&=e^{2\lambda}\sumzd\sum_{k=1}^de^{2\lambda|j|^2}|u_{j+e_k}(t)-u_{j-e_k}(t)|^2-2(e^{4\lambda}-1)\sumzd\sum_{k=1}^d \cosh(4\lambda j_k)e^{2\lambda |j|^2}|u_j(t)|^2,
\end{aligned}\]
we also have the bound
\begin{equation}\label{ap2}
\int_0^1t(1-t)\sumzd\sum_{k=1}^d e^{2\lambda|j|^2}|u_{j+e_k}(t)-u_{j-e_k}|^2\,dt\le c(G(1)+G(0)).
\end{equation}

In the continuous case, in order to conclude Hardy's uncertainty principle, this estimate for the gradient of the solution was crucial. However, here we can avoid the use of this estimate as we will see in Section 4, although it should be useful if one wants to relate this discrete result to the continuous one. Considering the modified Bessel functions, we can get similar estimates to those explained here for $e^{\lambda|j|^2}$. Notice that in a $\ell^\infty$ setting, Proposition 2.1 and Proposition 2.2 imply extra decay in the interior of $[0,1]$ for the solution as well.

\section{Log-convexity properties for solutions to perturbed discrete Schr\"odinger equations}

When we introduce a potential in the discrete Schr\"odinger equation, we cannot use the method we use in the previous Section in order to justify the calculations. However, we can first prove a log-convexity property for solutions to \eqref{eqschpot} where $V\equiv(V_j(t))_{j\in\mathbb{Z}^d}$ is a time-dependent bounded potential, and the solutions satisfy
\begin{equation}\label{condpesolin}
\sumzd e^{2j\cdot\lambda}\big(|u_j(0)|^2+|u_j(1)|^2\big)<+\infty . 
\end{equation}

\begin{lem}
Assume that $u$ is a solution to \eqref{eqschpot} where $V$ is a time-dependent bounded potential. Then, for $t\in[0,1]$ and $\beta\in\mathbb{R}$ we have
\[
\sumzd e^{2\beta j_1}|u_j(t)|^2\le e^{C\|V\|_\infty}\sumzd e^{2\beta j_1}\big(|u_j(0)|^2+|u_j(1)|^2\big),
\]
where $C$ is independent of $\beta$.
\end{lem}

\begin{rem}
$\|V\|_\infty$ stands for $\sup_{j\in\mathbb{Z}^d,t\in[0,1]}\{|V_j(t)|\}.$
\end{rem}

\begin{proof}
We are going to assume, without loss of generality that $\beta>0$. In order to give a rigorous proof of the result, we are going to truncate properly, following the procedure in \cite{kpv}, the weight $e^{\beta j_1}$ so that all the quantities that we are going to compute later on are valid and finite. To do this, we consider a function $\varphi \in C^\infty(\mathbb{R})$ such that $0\le \varphi\le 1$ and
\[
\varphi(x)=\left\{\begin{array}{ll}1,&x\le1,\\0,&x\ge 2.\end{array}\right.
\]  

Now, for $N\in\mathbb{N}$ we define $\varphi^N(x)=\varphi\left(\frac{x}{N}\right)$ and $\theta^N(s)=\beta \int_0^s(\varphi^N(y))^2dy$, so $\theta^N\in C^\infty(\mathbb{R})$ is non-decreasing and
\[
\theta^N(s)=\left\{\begin{array}{ll}\beta s,&s\le N,\\c\beta,&s> 2N.\end{array}\right.
\]

Moreover, we have that $|(\theta^N)'(s)|\le \beta$ and $|(\theta^N)''(s)|\le \frac{2\beta}{N}$. Finally, we discretize $\theta^N(s)$ by considering its evaluation at $\mathbb{Z}$. In other words, $\theta_{j_1}^N=\theta^N(j_1)$, for $j_1\in\mathbb{Z}.$ Notice that $\theta_{j_1}^N\uparrow \beta j_1$ as $N\rightarrow \infty.$ Now we take $f_j=e^{\theta_{j_1}^N}u_j(t)$ and compute the operators $S_N$ and $A_N$, symmetric and skew-symmetric respectively such that $\partial_tf_j=S_N f_j+A_Nf_j+iV_jf_j$. We have that
\[\begin{aligned}
\langle[S_N,A_N]&f_j,f_j\rangle=-2\Re\sumzd\sinh(\theta_{j_1+1}^N-2\theta_{j_1}^N+\theta_{j_1-1}^N)f_{j+e_1}\overline{f_{j-e_1}}\\&+2\sumzd\big(\cosh(\theta_{j_1+1}^N-\theta_{j_1}^N)\sinh(\theta_{j_1+1}^N-\theta_{j_1}^N)-\cosh(\theta_{j_1}^N-\theta_{j_1-1}^N)\sinh(\theta_{j_1}^N-\theta_{j_1-1}^N)\big)|f_j|^2,
\end{aligned}\]
and we want to bound this quantity from below. To do that, we define $v_N(x)=\theta^N(x+1)-\theta^N(x)$ so that $|\theta_{j_1+1}^N-2\theta_{j_1}^N+\theta_{j_1-1}^N|=|v_N(j_1)-v_N(j_1-1)|\le |v_N'(\xi)|\le |(\theta^N)''(\xi_1)|\le \frac{C\beta}{N}$ and $ \sinh(\theta_{j_1+1}^N-2\theta_{j_1}^N+\theta_{j_1-1}^N)\ge -\sinh\left(\frac{C\beta}{N}\right).$ On the other hand, the factor that appears in the second sum is $|\cosh(v_N(j_1))\sinh(v_N(j_1))-\cosh(v_N(j_1-1))\sinh(v_N(j_1-1)|\le \cosh(2\xi)|v_N(j_1)-v_N(j_1-1)|\le \frac{C\beta\cosh(2\beta)}{N}$, since $|\xi|\le \max\{|v_N(j_1-1)|,|v_N(j_1)|\}\le \beta.$ Combining these estimates we bound the commutator by
\[
\langle[S_N,A_N]f_j,f_j\rangle\ge-\left(\sinh\left(\frac{C\beta}{N}\right)+\frac{C\beta}{N}\cosh(2\beta)\right)\sumzd|f_j|^2.
\]

Now we use the perturbative version of Lemma 2.1 (see \cite{ekpv2}), taking into account that $|\partial_tf_j-S_Nf_j-A_Nf_j|=|V_j||f_j|\le\|V\|_\infty|f_j|$. Thus,
\[\begin{aligned}
\sumzd e^{2\theta_{j_1}^N}|u_j(t)|^2&\le e^{C(\sinh(C\beta/N)+\beta\cosh(2\beta)/N+\|V\|_\infty)}\sumzd e^{2\theta_{j_1}^N}\big(|u_j(0)|^2+|u_j(1)|^2\big)\\&\le e^{C(\sinh(C\beta/N)+\beta\cosh(2\beta)/N+\|V\|_\infty)}\sumzd e^{2\beta j_1}\big(|u_j(0)|^2+|u_j(1)|^2\big).
\end{aligned}\]

Using Fatou's lemma, we conclude the result.
\end{proof}

\begin{rem}
Using the same method, it is straightforward to see now that, for $\lambda\in\mathbb{R}^d,$
\begin{equation}\label{pesolin}
\sumzd e^{2\lambda\cdot j}|u_j(t)|^2\le e^{C\|V\|_\infty}\sumzd e^{2\lambda\cdot j}\big(|u_j(0)|^2+|u_j(1)|^2\big),
\end{equation}
\end{rem}
Once we have proved the lemma, it is quite easy to see that the log-convexity properties of Theorem 2.1, 2.2 and 2.3 are also satisfied when we add the potential $V$ to the equation. Notice that this fact gives another proof of those theorems, just by setting the potential equal to 0, although in this way we lose the a priori estimates \eqref{ap1}, \eqref{ap2}. Let us see how we can do this in the case of Theorem 2.2.

\begin{teo}
Assume $u=(u_j)_{j\in\mathbb{Z}^d}$ is a solution to the equation \eqref{eqschpot} where $V$ is a time-dependent bounded potential. Then, for $\alpha>0$ and $t\in[0,1]$,
\[
\sumzd\prod_{k=1}^d K_{j_k}^2(\alpha)|u_j(t)|^2 \le e^{c\|V\|_\infty}\sumzd\prod_{k=1}^dK_{j_k}^2(\alpha)\big(|u_j(0)|^2+|u_j(1)|^2\big),
\]
provided that the right-hand side is finite.
\end{teo}

\begin{proof}
When the right-hand side is finite, we have that $\sum_{j}e^{(\lambda_1+\lambda_2)\cdot j}\big(|u_j(0)|^2+|u_j(1)|^2\big)$ is finite $\forall\ \lambda_1,\lambda_2\in\mathbb{R}^d$, so applying \eqref{pesolin} we have
\[
\sumzd e^{(\lambda_1+\lambda_2)\cdot j}|u_j(t)|^2\le e^{c\|V\|_\infty}\sumzd e^{(\lambda_1+\lambda_2)\cdot j}\big(|u_j(0)|^2+|u_j(1)|^2\big),
\]
having that $(\lambda_1+\lambda_2)\cdot j=\sum_k (\lambda_{1,k}+\lambda_{2,k})\cdot j_k$. If we multiply this expression by $e^{-\alpha\sum_{k}(\cosh\lambda_{1,k}+\cosh\lambda_{2,k})}$ and integrate it in $(\lambda_1,\lambda_2)\in\mathbb{R}^{2d},$ the theorem holds using that
\[
\int_{\mathbb{R}}e^{\lambda j -\alpha \cosh\lambda}\,d\lambda=c K_j(\alpha).
\]\end{proof}
In order to prove the same log-convexity properties for the other weights we can do the same, now using that
\[
\sqrt{2\pi}e^{2\alpha j^2}=\int_{\mathbb{R}}e^{2\sqrt{\alpha}\lambda j-\lambda^2/2}\,d\lambda,
\]
while, in order to get the inverse of the modified Bessel function we do not have an explicit formula, but it can be checked that multiplying the linear exponential by a similar function that we have used in the proof of Theorem 3.1 we get that the integral behaves asymptotically in the same way as the inverse of the modified Bessel function $\frac{1}{I_j(\alpha)}$, whose asymptotic behavior is described in \eqref{asymb}.
\section{Hardy's uncertainty principle: Real variable approach}
In this Section we are going to give a discrete version of Hardy's principle. As we have pointed out above, among the weights we have studied here we are going to follow the approach in \cite{ekpv2} considering the function $e^{\lambda |j|^2}$. In order to do that we need the following Carleman inequality:

\begin{lem} The inequality
\[\begin{aligned}
R\sqrt{\frac{\epsilon}{8\mu}}&\|e^{\mu|j+Rt(1-t)e_1|^2-(1+\epsilon)R^2t(1-t)/16\mu}g\|_{L^2([0,1],\ell^2(\mathbb{Z}^d))}\\
&\le \|e^{\mu|j+Rt(1-t)e_1|^2-(1+\epsilon)R^2t(1-t)/16\mu}(\partial_t-i\Delta_d)g\|_{L^2([0,1],\ell^2(\mathbb{Z}^d))},
\end{aligned}\]
holds, for $R$ large enough, when $\epsilon>0, \mu>0$ and $g\in C_0^\infty(\mathbb{Z}^d\times[0,1]).$
\end{lem}

\begin{proof}

Consider $f_j(t)=e^{\mu|j+Rt(1-t)e_1|^2-(1+\epsilon)R^2t(1-t)/16\mu}g_j(t)$. Then, we can compute the operators $\mathcal{S}$ and $\mathcal{A}$ symmetric and skew-symmetric respectively, such that 
\[e^{\mu|j+Rt(1-t)e_1|^2-(1+\epsilon)R^2t(1-t)/16\mu}(\partial_t-i\Delta_d)g_j=\partial_tf_j-\mathcal{S}f_j-\mathcal{A}f_j.\]

Then, since
\[
\|\partial_tf-\mathcal{S}f-\mathcal{A}f\|_{L^2([0,1],\ell^2(\mathbb{Z}^d))}^2\ge \int_0^1\langle \mathcal{S}_tf+[\mathcal{S},\mathcal{A}]f,f\rangle dt,
\]
what we need to prove is $\langle \mathcal{S}_tf+[\mathcal{S},\mathcal{A}]f,f\rangle \ge \frac{\epsilon R^2}{8\mu}\|f\|_{\ell^2(\mathbb{Z}^d)}.$ After computing the commutator, we see that
\[\begin{aligned}
\langle \mathcal{S}_tf&+[\mathcal{S},\mathcal{A}]f,f\rangle=\sinh(2\mu)\sumzd|f_{j+e_1}-f_{j-e_1}|^2+32\mu^3\sumzd \left(j_1+R t(1-t)-\frac{R}{16\mu^2}\right)^2|f_j|^2\\
&+2\sinh(2\mu)\sumzd\big(\cosh(4\mu(j_1+Rt(1-t)))-1\big)|f_j|^2-32\mu^3\sumzd\big(j_1+Rt(1-t)\big)^2|f_j|^2\\
&+2\mu R^2(1-2t)^2\sumzd|f_j|^2+8\mu R(1-2t)\sumzd\cosh(2\mu(j_1+1/2+Rt(1-t)))\Im(f_{j+e_1}\overline{f_j})\\
&\sinh(2\mu)\sumzd\sum_{k=2}^d|f_{j+e_k}-f_{j-e_k}|^2+2\sinh(2\mu)\sumzd\sum_{k=2}^d \big(\cosh(4\mu j_k)-1\big)|f_j|^2+\frac{\epsilon R^2}{8\mu}\sumzd|f_j|^2.
\end{aligned}\]

Thus, the commutator can be studied separately, according to the behavior in each variable $j_k$, and the lemma holds if we prove that the following expression, written conveniently in one dimension,
\begin{equation}\label{eqc}\begin{aligned}
&32\mu^3\sumzd \left(j+R t(1-t)-\frac{R}{16\mu^2}\right)^2|f_j|^2+2\mu R^2(1-2t)^2\sum_j|f_j|^2\\
&+2\sinh(2\mu)\sum_j\big(\cosh(4\mu(j+Rt(1-t)))-1\big)|f_j|^2-32\mu^3\sum_j\big(j+Rt(1-t)\big)^2|f_j|^2\\
&+8\mu R(1-2t)\sum_j\cosh(2\mu(j+1/2+Rt(1-t)))\Im(f_{j+1}\overline{f_j})
\end{aligned}\end{equation}
is positive, since, going back to the expression for the commutator, this proves that the first three lines are positive, while the first two terms in the last line is clearly positive. By symmetry, the procedure will be the same if either $t\in[0,\frac12]$ or $t\in[\frac12,1]$, so we assume $t\le\frac12$. What we are going to see is that when $R$ is large, the leading terms in the expression for a general $j$ make the sum positive. We have to distinguish several regions and possibilities in order to give the whole proof of the lemma. First, let us assume that $t\ge \frac12-\frac{1}{16\mu^{3/2}},\ \ \mu>\mu_0$, where  $\mu_0$ is the positive value such that $\cosh^2(\mu_0)=2\sinh(2\mu_0)$, having that $\mu_0\simeq 0.255413\dots$. In this case we use $2\Im(f_{j+1}\overline{f_j})\ge -|f_{j+1}|^2-|f_j|^2$ to bound the expression by
\[\begin{aligned}
&32\mu^3\sum_j \left(j+R t(1-t)-\frac{R}{16\mu^2}\right)^2|f_j|^2+2\mu R^2(1-2t)^2\sum_j|f_j|^2\\
&+4\sinh(2\mu)\sum_j\big(\cosh^2(2\mu(j+Rt(1-t)))-1\big)|f_j|^2-32\mu^3\sum_j\big(j+Rt(1-t)\big)^2|f_j|^2\\
&-8\mu\cosh(\mu) R(1-2t)\sum_j\cosh(2\mu(j+Rt(1-t)))|f_j|^2,
\end{aligned}\]
so in this way we have an expression of the form $F(\mu,R,t,j)|f_j|^2$, and we would like to have that $F$ is positive. If $j$ is given such that $\cosh(2\mu(j+Rt(1-t)))>\frac{\cosh\mu}{4\mu^{1/2}\sinh(2\mu)}R,$ then, the leading terms of $F$, when $R$ is large, are
\[\begin{aligned}
4\sinh(2\mu)&\cosh^2(2\mu(j+Rt(1-t)))-8\mu\cosh(\mu) R(1-2t)\cosh(2\mu(j+Rt(1-t)))\\&>\cosh(2\mu(j+Rt(1-t)))R\left(\frac{\cosh\mu}{\mu^{1/2}}-\frac{\cosh\mu}{\mu^{1/2}}\right)=0.
\end{aligned}\]

On the other hand, if $j$ is given such that $\cosh(2\mu(j+Rt(1-t)))<\frac{1}{8\mu^{1/2}\cosh\mu}R$, then, since this implies that $j+Rt(1-t)$ is much less than $\frac{R}{16\mu^2}$, we see that the sign of $F$ is determined by
\[
\frac{R^2}{8\mu}-8\mu\cosh(\mu) R(1-2t)\cosh(2\mu(j+Rt(1-t)))>0.
\]

Hence, we have to consider now $j$ such that $\cosh(2\mu(j+Rt(1-t)))=\alpha R$, with $\alpha$ between the two values of the previous conditions. But in this case the leading terms in the expression are written as
\[
R^2\left(\frac{1}{8\mu}+4\sinh(2\mu)\alpha^2-\frac{\cosh\mu}{\mu^{1/2}}\alpha\right),
\]
and this is a polynomial of degree two in $\alpha$, whose discriminant is negative since $\mu>\mu_0$, so there is no change of sign in the polynomial and it is positive. This completes the proof when $t\ge \frac12-\frac{1}{16\mu^{3/2}}$ and $\mu>\mu_0$.

When $t\le \frac12-\frac{1}{16\mu^{3/2}},\ \mu>\mu_0$, we use the same reasoning to show that if $j$ is given such that
\begin{equation}\label{regj}
\cosh(2\mu(j+Rt(1-t)))<\frac{16\mu^2+1}{64\mu^2\cosh\mu}R(1-2t),\ \text{ or }\ \cosh(2\mu(j+Rt(1-t)))>\frac{4\mu\cosh\mu}{\sinh(2\mu)}R(1-2t),
\end{equation}
then the expression $F$ is positive for $R$ large enough. On the other hand, when $j$ is not in that region, we cannot use the same argument as before, because now we can have negative terms. However, we can define $j_0$ and $j_0'$ such that
\begin{itemize}
\item $j_0+Rt(1-t)$ is positive and for the first time $\cosh(2\mu(j_0+Rt(1-t)))\ge\frac{16\mu^2+1}{64\mu^2\cosh\mu}R(1-2t).$
\item $j_0'+Rt(1-t)$ is negative and for the first time $\cosh(2\mu(j_0'+Rt(1-t)))\ge\frac{16\mu^2+1}{64\mu^2\cosh\mu}R(1-2t).$
\end{itemize}

We are going to see here that for $j=j_0$ we can absorb the negative terms in \eqref{eqc}. Furthermore, we can use the same argument to prove the same for $j=j_0+1,j_0+2,\dots$ until we go into the region defined by \eqref{regj}. Actually, it is easy to check that $j_0+3$ is in the region defined in \eqref{regj}. Moreover, the proof for $j=j_0',j_0'-1,\dots$ is the same.

Since $j_0-1$ is in \eqref{regj} we have that the part of \eqref{eqc} involving $f_{j_0}$ that we need to study is
\begin{equation}\label{eqj0}\begin{aligned}
&\left(32\mu^3\left(j_0+R t(1-t)-\frac{R}{16\mu^2}\right)^2+2\mu R^2(1-2t)^2-32\mu^3\big(j_0+Rt(1-t)\big)^2\right)|f_{j_0}|^2\\
&+\left(4\sinh(2\mu)\big(\cosh^2(2\mu(j_0+Rt(1-t)))-1\big)-4\mu R(1-2t)\cosh(2\mu(j_0-1/2+Rt(1-t)))\right)|f_{j_0}|^2\\
&+8\mu R(1-2t)\cosh(2\mu(j_0+1/2+Rt(1-t)))\Im(f_{j_0+1}\overline{f_{j_0}})\\&+2\sinh(2\mu)\big(\cosh^2(2\mu(j_0+1+Rt(1-t)))-1\big)|f_{j_0+1}|^2,
\end{aligned}\end{equation}
where we add a part of \eqref{eqc} involving $f_{j_0+1}$ in order to gain an extra positive term. Now first we see that if $\Im(f_{j_0+1}\overline{f_{j_0}})\ge0$, then \eqref{eqj0} is positive. Indeed, in this case we get rid of the last two lines in \eqref{eqj0}, and considering that $\cosh(2\mu(j_0+Rt(1-t)))=\alpha R(1-2t)$ we can write the leading term in \eqref{eqj0} as
\[
2\mu R^2(1-2t)^2+4\sinh(2\mu)\alpha^2 R^2(1-2t)^2-4\mu \alpha \frac{R^2(1-2t)^2}{\cosh\mu+\sinh\mu},
\]
since $\cosh(x+a)=\cosh x\cosh a+\sqrt{\cosh^2x-1}\sinh a\simeq \cosh x(\cosh a+\sinh a)\Rightarrow \cosh(x-a)\simeq \frac{\cosh x}{\cosh a+\sinh a}$. Again, computing the discriminant of this polynomial we see that the expression is positive.

On the other hand, if $\Im(f_{j_0+1}\overline{f_{j_0}})<0,$ there is $z_0\in\mathbb{C}$ such that $\Im(z_0)=\gamma_0>0$ and $f_{j_0+1}=-z_0f_{j_0}\Rightarrow \Im(f_{j_0+1}\overline{f_{j_0}})=-\gamma_0|f_{j_0}|^2,\ \ |f_{j_0+1}|^2\ge \gamma_0^2|f_{j_0}|^2$. By using this we can bound the leading term in \eqref{eqj0} by 
\[\begin{aligned}
\frac{R^2}{8\mu}&+2\mu R^2(1-2t)^2+4\sinh(2\mu)\alpha^2 R^2(1-2t)^2-4\mu \alpha \frac{R^2(1-2t)^2}{\cosh\mu+\sinh\mu}\\&-8\mu R^2(1-2t)^2\alpha (\cosh\mu+\sinh\mu)\gamma_0+2\sinh(2\mu)\alpha^2 R^2(1-2t)^2(\cosh(2\mu)+\sinh(2\mu))^2\gamma_0^2.
\end{aligned}\]

Since $1\ge 1-2t$ we can write this as $R^2(1-2t)^2$ times a polynomial expression in $\alpha$ and $\gamma$ of degree 2 in each parameter. We can write it as a polynomial of degree two in $\alpha$ and compute minus the discriminant, getting
\begin{equation}\label{eqg}\begin{aligned}
4\left(2\mu+\frac{1}{8\mu}\right)&\big(4\sinh(2\mu)+2\sinh(2\mu)(\cosh(2\mu)+\sinh(2\mu))^2\gamma_0^2\big)\\&-\left(8\mu\gamma_0(\cosh\mu+\sinh\mu)+\frac{4\mu}{\cosh\mu+\sinh\mu}\right)^2.
\end{aligned}\end{equation}

We want the latter to be positive, and this is another polynomial, now in $\gamma_0$, of degree two, so we compute again the discriminant of the polynomial and we get
\[\begin{aligned}
4096\mu^4-4&\left(32\mu\sinh(2\mu)-\frac{16\mu^2}{(\cosh\mu+\sinh\mu)^2}+\frac{2\sinh(2\mu)}{\mu}\right)\bigg(16\mu\sinh(2\mu)(\cosh(2\mu)+\sinh(2\mu))^2\\&+\frac{\sinh(2\mu)}{\mu}(\cosh(2\mu)+\sinh(2\mu))^2-64\mu^2(\cosh\mu+\sinh\mu)^2\bigg),
\end{aligned}\]
which is easy to see that is negative. Hence, we have that \eqref{eqg} is positive and therefore  the part of \eqref{eqc} we are studying is positive. As we have pointed out above, we can repeat this argument for the rest of coefficients that we have to check to see that \eqref{eqc} is positive in this case.

It only remains the study of the case $\mu<\mu_0$. In order to proof Theorem 4.1 below, we do not need this case, since we are going to assume $\mu>1/2>\mu_0$. Nevertheless, it is easy to see that we can follow the following sketch to see that in this case \eqref{eqc} is also positive.
\begin{itemize}
\item Assume $j$ such that $\cosh(2\mu(j+Rt(1-t)))\le \frac{R}{200}$. After identifying the leading terms in $\eqref{eqc}$ for that $j$, we get $R^2$ times a polynomial of degree two in the variable $(1-2t)$, whose discriminant is negative when $\mu^2\cosh^2\mu<625$. Since $\mu<\mu_0$ we are in this case and the negativity of the discriminant implies that the leading term is positive.
\item If $\cosh(2\mu(j+Rt(1-t)))\ge R$ then we can prove the positivity of the leading term using $4\sinh(2\mu)\ge 8\mu\cosh\mu$.
\item Finally, if $\cosh(2\mu(j+Rt(1-t)))= \alpha R$ with $\alpha\in\left[\frac{1}{200},1\right]$, the leading term leads again to a polynomial expression of degree two, which is negative when
\[
64\mu^2\cosh^2\mu-16\sinh(2\mu)\left(\frac{1}{8\mu}+2\mu\right)<0,
\]
and this happens when $\mu<\mu_0$. Again, the latter implies that the leading term is positive.\end{itemize}\end{proof}

After proving this Carleman inequality, we can give a version of Hardy's uncertainty principle in the discrete setting using the weight $e^{\lambda |j|^2}$ as the discrete version of Gaussian decay, following the procedure used in \cite{ekpv2}.

\begin{teo}
Assume $u=(u_j)_{j\in\mathbb{Z}^d}$ is a solution to the equation \eqref{eqschpot}, where $V$ is a time dependent bounded potential, which satisfies the decay condition \eqref{decaygauss}, for some $\lambda>1/2$. Then $u\equiv0.$
\end{teo}

\begin{proof}
As we have seen in Section 3, the decay condition \eqref{decaygauss} implies that there is $N_{\lambda,V}$ depending only on $\lambda$ and $\|V\|_\infty$ such that
\[
\sup_{t\in[0,1]}\sumzd e^{2\lambda |j|^2}|u_j(t)|^2<N_{\lambda,V}.
\]

For given $R>0$, we take $\mu$ and $\epsilon$ such that
\[
\frac{(1+\epsilon)^{3/2}}{2(1-\epsilon)^3}<\mu\le \frac{\lambda}{1+\epsilon},
\]
and we consider, for $R<M\in\mathbb{N}$,  smooth functions $0\le \theta^M(x),\eta_R(t) \le 1$ such that
\[
\theta^M(x)=\left\{\begin{array}{ll}1,&|x|\le M,\\ 0,&|x|>2M.\end{array}\right.,\ \eta_R(t)=\left\{\begin{array}{ll}1,& t\in\left[\frac{1}{R},1-\frac{1}{R}\right],\\ 0,& t\in\left[0,\frac{1}{2R}\right]\cup\left[1-\frac{1}{2R},1\right].\end{array}\right.
\]

Then we define $\theta_j^M=\theta^M(j)$ for $j\in\mathbb{Z}^d$ and $g_j(t)=\theta_j^M\eta_R(t)u_j(t)$, so $g_j(t)$ is compactly supported in $\mathbb{Z}^d\times(0,1)$ and satisfies
\[
\partial_tg_j-i(\Delta_d g_j+V_jg_j)=\eta_R'\theta_j^Mu_j(t)-i\sum_{k=1}^d\big((\theta_{j_k+1}^M-\theta_{j_k}^M)\eta_R u_{j+e_k}-(\theta_{j_k}^M-\theta_{j_k-1}^M)\eta_R u_{j-e_k}\big).
\]

We apply Lemma 4.1 to show that there is a constant $N_\epsilon$ such that
\begin{equation}\label{carle}\begin{aligned}
R^2 \int_0^1\sum_j &e^{2\mu|j+Rt(1-t)e_1|^2-(1+\epsilon)R^2t(1-t)/8\mu}|g_j|^2\le N_\epsilon \|V\|_\infty^2\int_0^1\sum_j e^{2\mu|j+Rt(1-t)e_1|^2-(1+\epsilon)R^2t(1-t)/8\mu}|g_j|^2\\&+N_\epsilon\int_0^1\sum_j e^{2\mu|j+Rt(1-t)e_1|^2-(1+\epsilon)R^2t(1-t)/8\mu}|\eta_R'\theta_j^Mu_j|^2\\&+N_\epsilon\int_0^1\sum_j\sum_{k=1}^d e^{2\mu|j-e_k+Rt(1-t)e_1|^2-(1+\epsilon)R^2t(1-t)/8\mu}\eta_R^2|\theta_j^M-\theta_{j-e_k}^M|^2|u_j|^2\\&+N_\epsilon\int_0^1\sum_j\sum_{k=1}^d e^{2\mu|j+e_k+Rt(1-t)e_1|^2-(1+\epsilon)R^2t(1-t)/8\mu}\eta_R^2|\theta_{j+e_k}^M-\theta_{j}^M|^2|u_j|^2.
\end{aligned}\end{equation}

Now we have to understand each term of the right-hand side of \eqref{carle} separately. The first term can be hidden in the left-hand side by taking $R$ large enough. For the second term, it is supported where $-2M\le j\le 2M$ and $t\in\left[\frac{1}{2R},\frac{1}{R}\right]\cup\left[1-\frac{1}{R},1-\frac{1}{2R}\right]$, so, since $2ab\le \epsilon a^2+b^2/\epsilon$,
\[
\mu|j+Rt(1-t)e_1|^2\le \mu(1+\epsilon)|j|^2+\mu\left(1+\frac{1}{\epsilon}\right)\le \lambda |j|^2+\frac{\lambda}{\epsilon},
\]
while, for the other two regions, the support in $t$ and $j$ changes, but the same reasoning leads to
\[\begin{aligned}
&\mu|j-e_k+Rt(1-t)e_1|^2\le \lambda |j|^2+\frac{\lambda}{\epsilon}(Rt(1-t)+1)^2\le \lambda |j|^2+\frac{c\lambda R^2}{\epsilon},\\
&\mu|j+e_k+Rt(1-t)e_1|^2\le \lambda |j|^2+\frac{\lambda}{\epsilon}(Rt(1-t)+1)^2\le \lambda |j|^2+\frac{c\lambda R^2}{\epsilon},
\end{aligned}\]
with $c$ independent of $R,\lambda$, and $\epsilon$. Thus, using the natural bounds for the functions $\theta^M$ and $\eta_R$ we get
\begin{equation}\begin{aligned}
R^2 \int_0^1\sum_j &e^{2\mu|j+Rt(1-t)e_1|^2-(1+\epsilon)R^2t(1-t)/8\mu}|g_j|^2 \\&\le N_\epsilon(R^2 e^{2\lambda/\epsilon}+e^{c\lambda R^2/\epsilon}M^{-2})\sup_{t\in[0,1]}\sum_j e^{2\lambda |j|^2}|u_j(t)|^2. 
\end{aligned}\end{equation}

Furthermore, the term multiplied by $M^{-2}$ tends to zero when $M$ tends to infinity. On the other hand, we can bound the left-hand side just integrating in $t\in[(1-\epsilon)/2,(1+\epsilon)/2]$ and considering that $|j|\le \epsilon(1-\epsilon)^2R/4$. Under these circumstances, $g_j=u_j$ and moreover
\[
\mu|j+Rt(1-t)e_1|^2-(1+\epsilon)R^2t(1-t)/16\mu\ge \frac{R^2}{64\mu^2}\big(4\mu^2(1-\epsilon)^6-(1+\epsilon)^3\big)>0,
\]
so there is $C(\lambda,\epsilon)>0$ and $N_{\lambda,V,\epsilon}$ only depending on $\lambda,\|V\|_\infty$ and $\epsilon$ such that
\[
e^{C(\lambda,\epsilon)R^2}\|u\|_{L^2\big(\left(\frac{1-\epsilon}{2},\frac{1+\epsilon}{2}\right),\ell^2\left(|j|\le \epsilon(1-\epsilon)^2R/4\right)\big)}\le N_{\lambda,V,\epsilon}.
\]

Combining the fact that $N^{-1}\|u(0)\|_{\ell^2}\le\|u(t)\|_{\ell^2}\le N\|u(0)\|_{\ell^2}$ for $0\le t\le1$ and $N=e^{\|\Im V\|_\infty}$ with
\[
\|u(t)\|_{\ell^2(\mathbb{Z}^d)}\le \|u(t)\|_{\ell^2(|j|\le \epsilon(1-\epsilon)^2R/4)}+e^{-\lambda \epsilon^2(1-\epsilon)^4 R^2/16}N_{\lambda,V},
\]
we conclude that
\[
\|u(0)\|_{\ell^2(\mathbb{Z}^d)}\le N_{\lambda,V,\epsilon}(e^{-C(\lambda,\epsilon)R^2}+e^{-\lambda \epsilon^2(1-\epsilon)^4 R^2/16}).
\]

Finally, we let $R$ tend to infinity to have that $u\equiv0$.
\end{proof}
According to \eqref{hardyd} and \cite{jlmp}, this is not the best result we should expect. In the one dimensional setting, the sharp results assume decay conditions of the type $\sum_{j}\frac{|u_j(t_i)|^2}{I_j^2(\alpha)}$, with $\alpha<1$, $t_i\in\{0,1\}$, so it is reasonable to think that in the multidimensional setting the sharp condition is similar, now changing the modified Bessel function $I_j(\alpha)$ with the product of modified Bessel functions $\prod_{k=1}^dI_{j_k}(\alpha)$.

\section{Acknowledgments.} This paper is part of the Ph. D. thesis of the author, who is supported by the predoctoral grant BFI-2011-11 of the Basque Government and the projects MTM2011-24054, IT641-13. The author would like to thank L. Vega, without whose help this paper would not have been possible.%, and the reviewers for their constructive comments that have improved the paper.

%In that particular case, if we have a function $f_h$ that satisfies the hypothesis, and we use its Fourier coefficients as the initial datum of the discrete Schr\"odinger equation, then the solution is going to satisfy again the hypothesis. On the other hand, if we only assume that the Fourier coefficients decay as the modified Bessel function. Then, when we solve the discrete Schr\"odinger equation, we can reach a time where we can take a $\delta$ such that the solution satisfies all the hypothesis.
\end{document}